\documentclass[preprint]{elsarticle}
\usepackage{color}
\usepackage{amsmath}
\usepackage{amsthm}
\usepackage{amsfonts}
\usepackage{amssymb}
\usepackage{mathrsfs}
\usepackage{tabularx}

\usepackage{tikz}
\usetikzlibrary{shapes.misc}

\newcommand{\undertilde}[1]{\ensuremath{\mathord{\vtop{\ialign{##\crcr
   $\hfil\displaystyle{#1}\hfil$\crcr\noalign{\kern1.5pt\nointerlineskip}
   $\hfil\tilde{}\hfil$\crcr\noalign{\kern1.5pt}}}}}}

\tikzstyle{smallgraph}=[
  every node/.style={circle,draw,fill=black,inner sep=0pt,minimum width=0pt},
  every path/.append style={semithick}
]
\pgfmathsetmacro{\threevradius}{0.375*sqrt(2/3)}
\newcommand\gridgraph[2]{%
\foreach \n in {1,...,#1} {\foreach \m in {1,...,#2} {\node (\n\m) at (\m,-0.5*\n) {};}}
\foreach \n in {1,...,#1} {\draw (\n1) -- (\n#2);}
\foreach \n in {1,...,#2} {\draw (1\n) -- (#1\n);}
}

\newcommand\inS[2]{
\tikzset{inS/.style={circle,draw,fill=black,inner sep=0pt,minimum width=8pt}}
\node[inS] at (#1,-0.5*#2) {};
}

\newcommand\inSa[2]{
\tikzset{inS/.style={circle,draw,fill=white,inner sep=0pt,minimum width=8pt}}
\node[inS] at (#1,-0.5*#2) {};
}

\newcommand\notS[2]{
\tikzset{inS/.style={cross out,draw,minimum size=8*(\pgflinewidth), inner sep=0pt, outer sep=0pt}}
\node[inS] at (#1,-0.5*#2) {};
}

\newcommand\notSa[2]{
\tikzset{inS/.style={strike out,draw,minimum size=8*(\pgflinewidth), inner sep=0pt, outer sep=0pt}}
\node[inS] at (#1,-0.5*#2) {};
}

\renewenvironment{proof}{\par {\sc {\bf Proof.}\hskip 5pt}}{\hfill \qed \par}
\newtheorem{thm}{Theorem}\newtheorem{lem}[thm]{Lemma}\newtheorem{cor}[thm]{Corollary}
\newtheorem*{exa}{Example}

\begin{document}

\title{The Secure Domination Number of Cartesian Products of Small Graphs with Paths and Cycles}
\author[1]{Michael Haythorpe\corref{cor1}%
}\ead{michael.haythorpe@flinders.edu.au}\author[1]{Alex Newcombe}\ead{alex.newcombe@flinders.edu.au}

\address[1]{Flinders University, 1284 South Road, Tonsley, SA 5042, Australia, Ph: +61 882012375, Fax: +61 882012904}

\cortext[cor1]{Corresponding author}

\begin{abstract}The secure domination numbers of the Cartesian products of two small graphs with paths or cycles is determined, as well as for M\"{o}bius ladder graphs. Prior to this work, in all cases where the secure domination number has been determined, the proof has either been trivial, or has been derived from lower bounds established by considering different forms of domination. However, the latter mode of proof is not applicable for most graphs, including those considered here. Hence, this work represents the first attempt to determine secure domination numbers via the properties of secure domination itself, and it is expected that these methods may be used to determine further results in the future.

%
\end{abstract}

\maketitle
\section{Introduction}

Consider an undirected graph $G$ containing vertex set $V$ and edge set $E$. Then $S \subseteq V$ is said to be a {\em dominating set} if, for every vertex $v \in V$, we have either $v \in S$, or there exists $w \in V$ such that $\{v,w\} \in E$ and $w \in S$. In the latter case, we say that $w$ {\em covers} $v$. The size of the smallest dominating set in $G$ is called the {\em domination number of $G$}, and is denoted by $\gamma(G)$. The {\em domination problem} is to determine the domination number of a given graph. The decision problem variant of this problem, asking whether a dominating set exists with cardinality no more than a given constant, is known to be NP-complete \cite{gareyjohnson}.

The domination problem has obvious real-world applications. A common example is to imagine a set of locations which must remain under observation. We could place a guard at each location; however, it may be possible for a guard at one location to see another location. We can represent this situation by a graph $G$ where the vertices correspond to the locations, and an edge $\{v,w\}$ exists if a guard at location $v$ can observe location $w$ as well. Then, for any dominating set $S$ for $G$, placing guards at the locations in $S$ ensures every location is under observation. Of course, it is likely to be desirable to do this with as few guards as possible.

In 2005, Cockayne et al \cite{cockayne} considered a new variant of the domination problem. They defined a {\em secure dominating set} $S \subseteq V$ to be one in which, for every vertex $v \in V$, we have either $v \in S$, or there exists $w \in V$ such that $w$ covers $v$, and $(S \setminus \{w\}) \cup \{v\}$ is a dominating set. In the latter case, we say that $w$ {\em guards} $v$. The interpretation is perhaps best illustrated by continuing the example above. Suppose that there is a disturbance at one of the locations. If a guard is present at the location, they deal with the disturbance. If not, then one of the guards from a neighbouring location must move to this location in order to deal with the disturbance. It is obviously desirable if the full set of locations is still under observation when this occurs. Then, the cardinality of the smallest secure dominating set in $G$ is called the {\em secure domination number of $G$}, denoted by $\gamma_s(G)$, and the {\em secure domination problem} is to determine the secure domination number of a given graph. The secure domination problem is also known to be NP-complete \cite{wang}.

Algorithms exist to solve the secure domination problem, including linear-time algorithms for trees \cite{burgerlinear} and block graphs \cite{pradhan}, as well as algorithms for general graphs \cite{burdett,burgerbnb,burgerformulation}. However, there are very few infinite families of graphs for which the secure domination number is known. When introducing the secure domination problem, Cockayne et al \cite{cockayne} determined the secure domination number for complete multipartite graphs (including complete graphs), paths, and cycles. In 2019, Valveny and Rodriguez-Vel\'{a}zquez \cite{valveny} expanded on this by determining the secure domination number for each of the following: the corona product of any graph with a discrete graph, the Cartesian product of two equal-sized stars, and the Cartesian product of any complete graph (other than $K_2$) with either a path, cycle, or star.

In all of these results, the proofs were either trivial (such as for complete graphs), or took advantage of general lower bounds. Most commonly, the lower bounds came from first considering the {\em weak Roman domination number}, which is itself a lower bound for the secure domination number. This is worth mentioning because none of the existing results were determined by considering the specific properties unique to secure dominating sets. Of course, this approach can only be fruitful for graph families for which the lower bound obtained from another type of domination happen to coincide with the secure domination number. In general, this is not the case.

The types of graphs considered in this work have been studied in the context of other types of domination. Notably, the domination number of the Cartesian product of two paths (a grid graph) has been extensively studied in \cite{Alanko,Chang,Dreyer} and has now been effectively solved by Gon\c{c}laves et al \cite{Gonclaves}. The domination number of the Cartesian product of other combinations of paths and cycles has been investigated in \cite{Dreyer,Klavzar,Nandi,Pavilic}. The Roman domination number of similar graphs has been investigated in \cite{Cockayne2,Dreyer,Pavilic2,Yero}. Although there are many results known about these graphs for other types of domination, none of these lead to the corresponding result for secure domination.

We now add to this literature by considering Cartesian products of $P_2$ and $P_3$ with arbitrarily large paths or cycles. We adopt the convention that $P_n$ and $C_n$ respectively refer to paths and cycles containing $n$ vertices. We also use the notation $M_{2n}$ to refer to a M\"{o}bius ladder graph with $2n$ vertices; these are only defined for an even number of vertices. Upper bounds for $\gamma_s(P_2 \Box P_n)$ and $\gamma_s(P_3 \Box P_n)$ were given by Cockayne et al \cite{cockayne}, while the conjectured value for $\gamma_s(P_2 \Box C_n)$ was given by Valveny and Rodriguez-Vel\'{a}zquez \cite{valveny} and separately by Winter \cite{winter}. Nobody has previously considered $P_3 \Box C_n$. We will determine the secure dominating number for each of these graph families, and as a corollary to the $P_2 \Box C_n$ case, we are also able to determine the secure domination number for $M_{2n}$. It is our expectation that these techniques will prove useful for determining the secure domination numbers of other infinite graph families as well.

\section{Notation and useful results\label{sec-notation}}

We begin by establishing some valuable results and techniques which will be useful during the proofs in the proceeding sections. Although there have been some important bounds established in previous work, these have mainly been on the secure domination number itself. For the sake of the upcoming proofs, it is valuable to establish results for a given secure dominating set. The first item in the following Lemma was pointed out by \cite{valveny} and has been utilised several times before.


\begin{lem}Suppose that $S$ is a secure dominating set for a graph $G$ containing vertex set $V$ and edge set $E$. Then if any combination of the following modifications is made to $G$, $S$ is still secure dominating for the resulting graph.\label{thm-mods}
\begin{enumerate}\item Adding a new edge between existing vertices in $V$.
\item Deleting an edge $vw$ for $v \in S$ and $w \in S$.
\item Deleting a vertex $v$ (and its incident edges) for $v \not\in S$.\end{enumerate}\end{lem}

\begin{proof}In order for $S$ to be secure dominating in a graph, the following must be satisfied. For every $v \not\in S$, there must be a $w \in S$ such that $vw \in E$, and $(S \setminus \{w\}) \cup \{v\} $ is dominating. In this situation, we will say that {\em $w$ guards $v$}.

Consider any vertex $v \not\in S$. In $G$, there exists a $w$ which guards $v$. This remains true even if another edge is added, and hence $S$ is still secure dominating.


Next, consider two vertices $v \in S$ and $w \in S$. Neither vertex requires another to guard them, and hence the deletion of $vw$ does not prevent $S$ from being secure dominating.

Finally, consider a vertex $v \not\in S$. If it is deleted, the requirement to find a vertex which guards $v$ is removed. Its incident edges are also deleted, however since $v \not\in S$, it could never have been the case that $v$ guarded another vertex. Hence, $S$ is still secure dominating.\end{proof}

Recall that the {\em open neighborhood} of a vertex $v$ is the set of vertices that are adjacent to $v$ and is denoted $N(v)$.

\begin{lem}Suppose that $S$ is a secure dominating set for a graph $G$. If an edge $vw$, with $v \notin S$ and $w \notin S$, satisfies at least one of the following
\begin{enumerate}\item $|N(v) \cap S|>1$ and $|N(w) \cap S|>1$;
\item $|N(v) \cap N(w) \cap S | \neq 1$,
\end{enumerate}
then, edge $vw$ may be deleted and $S$ is still secure dominating for the resulting graph. \label{thm-mods2}
\end{lem}

\begin{proof}
Consider any vertex $a \in S$ that guards $v$. Because $vw$ satisfies either 1. or 2. there must exist another vertex $a' \in N(w) \cap S$ where $a' \neq a$. Thus $S \setminus \{a\} \cup \{v\}$ is dominating in the graph with edge $vw$ deleted. The same argument can be applied to to any vertex that guards $w$ and thus $S$ is secure dominating in in the graph with edge $vw$ deleted.
\end{proof}

The following corollary emerges immediately from the third item in Lemma \ref{thm-mods}, along with the fact that $\gamma_s(G_1 \cup G_2) = \gamma_s(G_1) + \gamma_s(G_2)$.

\begin{cor}
Suppose that $S$ is a secure dominating set for a graph $G$ containing vertex set $V$, and there exists $U \subset V$ such that if an edge $ab$ exists for $a \in U$ and $b \not\in U$, then $a \not\in S$. Then if $G_2$ is constructed by deleting $U$ and its adjacent edges, $|S| \geq \gamma_s(G_2) + |U \cap S|$.\label{cor-trim}
\end{cor}

Since the results we will establish in this paper are for graphs which emerge from Cartesian products of $P_2$ or $P_3$, with another graph, the resulting graphs may be thought of as containing many copies of $P_2$ or $P_3$. For $P_2 \Box P_n$ and $P_3 \Box P_n$, it will sometimes be important to refer to specific copies. In such cases, we will use the notation $P_2^i$ and $P_3^i$ to refer to the $i$-th copy, for $i \in \{1,2,\dots,n\}$. In other cases, we simply want to consider a set of $m$ consecutive copies of arbitrary starting position. In such cases, we will use the notation $G^1, \hdots G^m$, and still refer to $G^i$ as the $i$-th copy within the set. In both cases, we will also use the notation $v^i_j$ to refer to the $j$-th vertex in the $i$-th copy, where $j=1$ refers to the top vertex (in the context of Cartesian products with paths) or the outer vertex (in the context of Cartesian products with cycles) in an embedding in the plane.

In the upcoming proofs, a common technique will be to consider how many vertices from each copy are contained in $S$. In general, we will say that if a copy contains $i$ vertices in $S$, then it is a {\em type $i$ copy}, and we will use $m_i$ to denote the number of type $i$ copies for $0 \leq i \leq 3$. Clearly, for $P_2 \Box P_n$ and $P_2 \Box C_n$, $m_3 = 0$ by definition. Then, it is clear that $m_0 + m_1 + m_2 + m_3 = n$, but also $m_1 + 2m_2 + 3m_3 = |S|$. Combining these, we obtain the below equation, which will be regularly used to establish bounds on $m_i$.

\begin{eqnarray}m_0 = n - |S| + m_2 + 2m_3.\label{eq-balance}\end{eqnarray}

We will use the term {\em block} of length $l$ to refer to a set of $l$ consecutive copies such that the first and last copies have type 0, and the other copies do not have type 0. Clearly, we can use the vertices of a block as the set $U$ in Corollary \ref{cor-trim}. Furthermore, it will sometimes be useful to take advantage of the following result.

\begin{lem}Suppose that $S$ is a secure dominating set for a graph $G \Box C_n$. Then if we have a block of length $l$ containing the vertices $U$, then $|S| \geq \gamma_s(G \Box C_{n-l}) + |U \cap S|$. Also, if we have a single type 0 copy, then $|S| \geq \gamma_s(G \Box C_{n-1})$.\label{lem-smoothing}\end{lem}

\begin{proof}We handle both cases simultaneously. As with Corollary \ref{cor-trim}, we can delete the edges which connect $U$ to the rest of the graph, and $S$ remains secure dominating for the new, disconnected graph. Then, from item 1 of Lemma \ref{thm-mods} we can add edges to create a new graph $G_2$ which is the union of the subgraph induced by the vertices of $U$, and a smaller Cartesian product with $l$ fewer copies of $G$. The result then follows immediately.\end{proof}

Lemma \ref{lem-smoothing} is a special case of a more general technique, where a subgraph satisfying the conditions of Corollary \ref{cor-trim} can be essentially smoothed, rather than simply deleted. We will use a specific application of this technique in Lemma \ref{lem-P3Cn1s}.

We will use the term {\em pattern} to refer to a set of consecutive copies which have specified types. For instance, if we say that $S$ contains the pattern 101, it implies that there is a set of three consecutive copies where the first is of type 1, the second is of type 0, and the third is of type 1. In many cases, we will establish that certain patterns cannot exist in $S$, usually taking advantage of symmetry to avoid considering unnecessarily many cases. In order to illustrate this, we provide a short example here, which also serves to introduce the graphical notation that will be used.

\begin{exa}Suppose $S$ is a secure dominating set for $P_2 \Box P_n$. We will show that $S$ cannot contain the pattern 1010. Suppose that it does. We consider $G^1, \hdots, G^4$, a set of four consecutive copies of $P_2$, obeying the pattern 1010. Consider $G^1$ first; due to symmetry, it does not matter which vertex is in $S$. Hence, we will assume that $v^1_1 \in S$. Then, $G_2$ is of type 0 and so by the domination condition, it is clear that $v^3_2 \in S$ (in order to cover $v^2_2$). Note that since $G^3$ is type 1, this also implies that $v^3_1 \not\in S$. Then, if there is an attack at $v^2_2$, the guard at $v^3_2$ must move to assist, but this leaves $v^3_1$ unguarded, because $G^4$ is also type 0 copy. Hence, this situation cannot occur, and so $S$ cannot contain the pattern 1010. This situation is displayed in Figure \ref{fig-example}. Note that we use a solid circle to denote vertices which are fixed in $S$ by assumption, and hollow circles to denote vertices which we subsequently argue must appear in $S$. Likewise, we use crosses to denote those vertices which are not in $S$ by assumption, and a strikeout to denote vertices which we subsequently argue must not appear in $S$.\label{example}\end{exa}

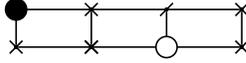
\begin{figure}[h!]\begin{center}\begin{tikzpicture}[smallgraph]\gridgraph{2}{4} \inS{1}{1} \inSa{3}{2} \notS{1}{2} \notS{2}{1} \notS{2}{2} \notS{2}{2} \notS{4}{1} \notS{4}{2} \notSa{3}{1}\end{tikzpicture}
\caption{The situation described in Example \ref{example}.\label{fig-example}}\end{center}\end{figure}

Finally, we note that each of the upcoming proofs will use induction. Hence, it will be necessary to establish the secure domination number for a certain number (depending on the proof) of small cases. Rather than present these small cases here, we will rely on the exact formulation given in Burdett and Haythorpe \cite{burdett} to handle these.

\section{Upper bounds\label{sec-upperbounds}}

In the proceeding sections, we will be establishing lower bounds. In order to obtain equalities, we will also require upper bounds, which we provide here.

\begin{thm}We have the following upper bounds:
\begin{eqnarray*}\gamma_s(P_2 \Box P_n) & \leq & \left\lceil \frac{3n+1}{4} \right\rceil\mbox{ for }n \geq 2,\\
\gamma_s(P_2 \Box C_n) & \leq & \frac{3n}{4}\mbox{ for }n = 0\mbox{ mod }8, n \geq 8,\\
\gamma_s(M_{2n}) & \leq & \frac{3n}{4}\mbox{ for }n = 4\mbox{ mod }8, n \geq 12,\\
\gamma_s(P_3 \Box P_n) & \leq & n+2\mbox{ for }n \geq 2,\\
\gamma_s(P_3 \Box C_n) & \leq & 3\left\lceil \frac{n}{3} \right\rceil\mbox{ for }n \geq 3.\end{eqnarray*}\end{thm}

\begin{proof}The upper bound for $P_3 \Box P_n$ was established in Cockayne et al \cite{cockayne}. For the other cases, we will provide drawings which corresponding to a secure dominating set $S$ which meets the upper bounds. Verifying that $S$ is secure dominating is left as an exercise to the reader.

For $P_2 \Box P_n$, the vertex patterns in the first and second items shown in Figure \ref{fig-ubP2Pn} can be repeated, alternating between them. Suppose that we do so, obtaining a vertex pattern for a graph with $4\left\lceil \frac{n}{4} \right\rceil$ vertices. We now describe how to modify this to obtain a valid $S$ for $P_2 \Box P_n$. If $n = 0\mbox{ mod }4$, simply add either of the vertices from the final copy to $S$. If $n = 1\mbox{ mod }4$, delete the first, and the final two copies. For $n = 2\mbox{ mod }4$, delete the final two copies. For $n = 3\mbox{ mod }4$, delete the final copy.

For $P_2 \Box C_n$, if $n = 0\mbox{ mod }8$, then the vertex pattern shown in Figure \ref{fig-ubP2Cn} can be repeated as many times as needed.

For $M_{2n}$, if $n = 4\mbox{ mod }8$, then the vertex pattern shown in Figure \ref{fig-ubM2n} can be followed by as many copies of the vertex pattern shown in Figure \ref{fig-ubP2Cn} as needed.

For $P_3 \Box C_n$, the first item in Figure \ref{fig-ubP3Cn} shows the vertex pattern which can be repeated as many times as needed. Suppose that we do so, obtaining a vertex pattern for a graph with $3\left\lfloor \frac{n}{3} \right\rfloor$ vertices. We now describe how to modify this to obtain a valid $S$ for $P_3 \Box C_n$. If $n = 0\mbox{ mod }3$, no modification is necessary. If $n = 1\mbox{ mod }3$, the second item in Figure \ref{fig-ubP3Cn} should be added as a final copy. If $n = 2\mbox{ mod }3$, the third item in Figure \ref{fig-ubP3Cn} should be added as the final two copies.\end{proof}

\begin{figure}[h!]\begin{center}
\begin{minipage}{0.3\textwidth}\begin{center}\begin{tikzpicture}[smallgraph]
\gridgraph{2}{4}
\inS{1}{1} \inS{2}{2} \inS{3}{1} \notS{1}{2} \notS{2}{1} \notS{3}{2} \notS{4}{1} \notS{4}{2}
\end{tikzpicture}\end{center}\end{minipage}\begin{minipage}{0.3\textwidth}\begin{center}\begin{tikzpicture}[smallgraph]
\gridgraph{2}{4}
\inS{1}{2} \inS{2}{1} \inS{3}{2} \notS{1}{1} \notS{2}{2} \notS{3}{1} \notS{4}{1} \notS{4}{2}
\end{tikzpicture}\end{center}\end{minipage}
\caption{The vertex patterns for $P_2 \Box P_n$, which can be repeated as many times as needed, alternating between them. The result can then be modified accordingly, as described above.\label{fig-ubP2Pn}}\end{center}\end{figure}
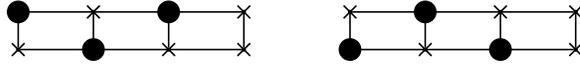

\begin{figure}[h!]\begin{center}
\begin{tikzpicture}[smallgraph]
\gridgraph{2}{8}
\inS{2}{1} \inS{3}{2} \inS{5}{1} \inS{5}{2} \inS{7}{1} \inS{8}{2} \notS{1}{1} \notS{1}{2} \notS{2}{2} \notS{3}{1} \notS{4}{1} \notS{4}{2} \notS{6}{1} \notS{6}{2} \notS{7}{2} \notS{8}{1}
\end{tikzpicture}
\caption{The vertex pattern which can be repeated as many times as needed for $P_2 \Box C_n$, $n = 0\mbox{ mod }8$.\label{fig-ubP2Cn}}\end{center}\end{figure}
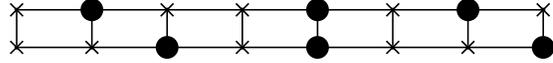

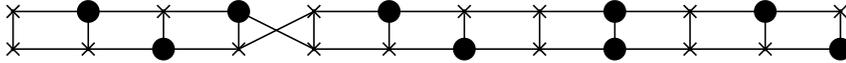
\begin{figure}[h!]\begin{center}
\begin{tikzpicture}[smallgraph]
\gridgraph{2}{4}
\foreach \n in {1,2} {\foreach \m in {5,...,12} {\node (\n\m) at (\m,-0.5*\n) {};}}
\foreach \n in {1,...,2} {\draw (\n5) -- (\n12);}
\foreach \n in {5,...,12} {\draw (1\n) -- (2\n);}
\inS{2}{1} \inS{3}{2} \inS{4}{1} \notS{1}{1} \notS{1}{2} \notS{2}{2} \notS{3}{1} \notS{4}{2}
\inS{6}{1} \inS{7}{2} \inS{9}{1} \inS{9}{2} \inS{11}{1} \inS{12}{2} \notS{5}{1} \notS{5}{2} \notS{6}{2} \notS{7}{1} \notS{8}{1} \notS{8}{2} \notS{10}{1} \notS{10}{2} \notS{11}{2} \notS{12}{1}
\draw (14) -- (25); \draw(15) -- (24);
\end{tikzpicture}
\caption{The vertex pattern for $M_{2n}$, $n = 4\mbox{ mod }8$, which may then be followed by as many copies of the vertex pattern from Figure \ref{fig-ubP2Cn} as needed.\label{fig-ubM2n}}\end{center}\end{figure}

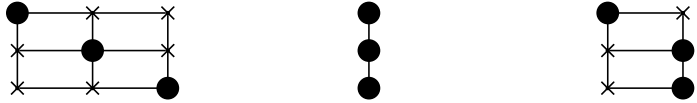
\begin{figure}[h!]\begin{center}
\begin{minipage}{0.25\textwidth}\begin{center}\begin{tikzpicture}[smallgraph]
\gridgraph{3}{3}
\inS{1}{1} \inS{2}{2} \inS{3}{3} \notS{1}{2} \notS{1}{3} \notS{2}{1} \notS{2}{3} \notS{3}{1} \notS{3}{2}
\end{tikzpicture}\end{center}\end{minipage}\begin{minipage}{0.25\textwidth}\begin{center}\begin{tikzpicture}[smallgraph]
\gridgraph{3}{1}
\inS{1}{1} \inS{1}{2} \inS{1}{3}
\end{tikzpicture}\end{center}\end{minipage}\begin{minipage}{0.25\textwidth}\begin{center}\begin{tikzpicture}[smallgraph]
\gridgraph{3}{2}
\inS{1}{1} \inS{2}{2} \inS{2}{3} \notS{1}{2} \notS{1}{3} \notS{2}{1}
\end{tikzpicture}\end{center}\end{minipage}
\caption{The vertex patterns for $P_3 \Box C_n$. The first item can be repeated as many times as needed. Then if $n = 1\mbox{ mod }3$, the second item is used as the final copy. If $n = 2\mbox{ mod }3$, the third item is used as the final two copies.\label{fig-ubP3Cn}}\end{center}\end{figure}

\section{The secure domination number of $P_2 \Box P_n$}

We now consider the graph $P_2 \Box P_n$ for $n \geq 2$. In Section \ref{sec-upperbounds}, we established that $\gamma_s(P_2 \Box P_n) \leq \left\lceil \frac{3n+1}{4} \right\rceil$ for $n \geq 2$. We will now prove that equality holds.

For small cases, we are able to use the exact formulation from \cite{burdett}, and indeed, we have done so to show that equality holds for $n \leq 16$. However, it could be the case that there is some value $k \geq 17$ such that equality holds for $n = 2, \hdots, k-1$, but $\gamma_s(P_2 \Box P_k) \leq \left\lceil \frac{3k-3}{4} \right\rceil = \left\lfloor \frac{3k}{4} \right\rfloor$. If so, then there exists a secure dominating set $S$ for $P_2 \Box P_k$ such that $|S| = \left\lfloor \frac{3k}{4} \right\rfloor.$ We will now show that this is impossible.


Recall that $P_2 \Box P_k$ contains $k$ copies of $P_2$, which we call $P_2^1, P_2^2, \hdots, P_2^k$. We will refer to $P_2^1$ and $P_2^k$ as ``end" copies, $P_2^2$ and $P_2^{k-1}$ as ``second-end" copies, and all other copies as ``internal copies". In this case, equation (\ref{eq-balance}) becomes:

\begin{equation}m_0 = \left\lceil \frac{k}{4} \right\rceil + m_2.\label{eq-P2Pnm0}\end{equation}

Note that, since $k \geq 17$, we have $m_0 \geq 5$. Hence, there is at least one internal copy of type 0.

\begin{lem}Suppose that $S$ is a secure dominating set for $P_2 \Box P_k$ such that $|S| = \left\lceil \frac{3k-3}{4} \right\rceil$. Furthermore, suppose that for all $n = 2, \hdots, k-1$, we have $\gamma_s(P_2 \Box P_n) = \left\lceil \frac{3n+1}{4} \right\rceil$. Then if $P_2^2$ is of type 0, it must be the case that $P_2^1$ is of type 1, and $P_2^3$ is of type 2. Similarly, if $P_2^{k-1}$ is of type 0, it must be the case that $P_2^k$ is of type 1, and $P_2^{k-2}$ is of type 2.\label{lem-P2PnG2}\end{lem}

\begin{proof}Suppose that $P_2^2$ is of type 0. In order to satisfy the domination condition, $P_2^1$ must either be of type 1 or type 2. Suppose it is of type 2, then we can delete $P_2^1$ and $P_2^2$, obtaining $P_2 \Box P_{k-2}$. Then, from Corollary \ref{cor-trim} we have $|S| \geq \left\lceil \frac{3k-5}{4} \right\rceil + 2$, which is a contradiction. Hence, $P_2^1$ is of type 1.

Then, due to symmetry, we can select either vertex of $P_2^1$ to be in $S$. Assume $v^1_1 \in S$. Then, in order to satisfy the domination property, $v^3_2 \in S$. Also, if there is an attack at $v^1_2$, the guard at $v^1_1$ must move to assist, which implies that $v^3_1 \in S$ in order to avoid leaving $v^2_1$ unguarded. Hence $P_2^3$ must be of type 2. This situation is displayed in Figure \ref{fig-P2PnG2}.

Due to symmetry, an analagous argument can be made for $P_2^{k-1}$, completing the proof.\end{proof}

\begin{figure}[h!]\begin{center}\begin{tikzpicture}[smallgraph]\gridgraph{2}{3} \inS{1}{1} \inSa{3}{1} \inSa{3}{2} \notS{1}{2} \notS{2}{1} \notS{2}{2}\end{tikzpicture}
\caption{The situation described in Lemma \ref{lem-P2PnG2}.\label{fig-P2PnG2}}\end{center}\end{figure}
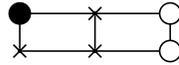

\begin{lem}Suppose that $S$ is a secure dominating set for $P_2 \Box P_k$. Then $S$ cannot contain the pattern 101101.\label{lem-P2Pn101101}\end{lem}

\begin{proof}Suppose that we have six consecutive copies of $P_2$, call them $G^1, \hdots, G^6$, obeying the pattern 101101. We denote $v^i_1$ and $v^i_2$ to be the top and bottom vertices of $G^i$, respectively. Due to symmetry, it doesn't matter which vertex of $G^1$ is in $S$. Assume $v^1_1 \in S$. By the domination condition, $v^3_2 \in S$. Then, if there is an attack at $v^2_2$, the guard at $v^3_2$ must move to assist, which implies that $v^4_1 \in S$ in order to avoid leaving $v^3_1$ unguarded. Finally, by the domination condition, $v^6_2 \in S$. This situation is displayed in Figure \ref{fig-P2Pn101101}. Now, suppose there is an attack at $v^3_1$. Either the guard at $v^3_2$, or the guard at $v^4_1$ must move to assist. In the former case, $v^2_2$ is left unguarded. In the latter case, $v^5_1$ is left unguarded. This implies that $S$ is not secure dominating, which is a contradiction, and hence $S$ cannot contain the pattern 101101.\end{proof}

\begin{figure}[h!]\begin{center}\begin{tikzpicture}[smallgraph]\gridgraph{2}{6} \inS{1}{1} \inSa{3}{2} \inSa{4}{1} \inSa{6}{2} \notS{1}{2} \notS{2}{1} \notS{2}{2} \notSa{3}{1} \notSa{4}{2} \notS{5}{1} \notS{5}{2} \notSa{6}{1}\end{tikzpicture}
\caption{The situation described in Lemma \ref{lem-P2Pn101101}.\label{fig-P2Pn101101}}\end{center}\end{figure}
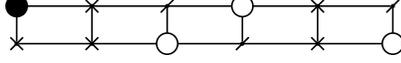

\begin{lem}Suppose that $S$ is a secure dominating set for $P_2 \Box P_k$ such that $|S| = \left\lceil \frac{3k-3}{4} \right\rceil$. Furthermore, suppose that for all $n = 2, \hdots, k-1$, we have $\gamma_s(P_2 \Box P_n) = \left\lceil \frac{3n+1}{4} \right\rceil$. Then, $k = 0\mbox{ mod }4$, and any internal copy $P_2^m$ can only be of type 0 if $m \mbox{ mod }4 \in \{2,3\}$.\label{lem-P2Pn0mod4}\end{lem}

\begin{proof}Recall from (\ref{eq-P2Pnm0}) that there is at least one internal copy of type 0. Consider such a copy, say $P_2^m$ for some $m \in \{3, \hdots, k-2\}$. If we delete this copy, we obtain the union of $P_2 \Box P_{m-1}$ and $P_2 \Box P_{k-m}$. Then, from Corollary \ref{cor-trim} we have $|S| \geq \left\lceil \frac{3m-2}{4} \right\rceil + \left\lceil \frac{3k-3m+1}{4} \right\rceil$. By checking all combinations of $k$ and $m$ mod 4, it can be seen that this is a contradiction except for the following two cases: either $k = 0\mbox{ mod }4$ and $m \in \{2,3\}\mbox{ mod }4$, or $k = 3\mbox{ mod }4$ and $m = 2\mbox{ mod }4$.

Suppose $k = 3\mbox{ mod }4$. Recall from (\ref{eq-P2Pnm0}) that $m_0 = \left\lceil \frac{k}{4} \right\rceil + m_2$. However, internal copies $P_2^m$ can only be of type 0 if $m = 2\mbox{ mod }4$. Hence, at most $\left\lceil \frac{k}{4} \right\rceil - 2$ internal copies are of type 0. Hence, $2 + m_2$ of $P_2^1, P_2^2, P_2^{k-1}, P_2^k$ must be of type 0. Suppose that $P_2^1$ is of type 0. Then, by the domination condition, $P_2^2$ must be of type 2, which implies that both $P_2^{k-1}$ and $P_2^k$ must be type 0, which violates the domination condition. Hence, $P_2^1$ is not of type 0. An analagous argument can be made to conclude that $P_2^k$ is not of type 0. Hence, both $P_2^2$ and $P_2^{k-1}$ must be of type 0, and $m_2 = 0$. However, according to Lemma \ref{fig-P2PnG2} this is impossible, and hence $k \neq 3\mbox{ mod }4$, completing the proof.\end{proof}

In the next proof, we utilise the concept of a ``block", defined in Section \ref{sec-notation}. We similarly define an ``end-block" to be a sequence of consecutive copies of $P_2$ such that the first is $P_2^1$ and the last is of type 0, or such that the first is of type 0 and the last is $P_2^k$. Also, the end-copy in an end-block (either $P_2^1$ or $P_2^k$) should not be type 0; if it is, then this is just a block, rather than an end-block. As such, an end-block contains only one copy of type 0, whereas a block contains two copies of type 0. As with a block, if the end-block contains $l$ copies of $P_2$, we say it has length $l$.

\begin{lem}Suppose that $S$ is a secure dominating set for $P_2 \Box P_k$ such that $|S| = \left\lceil \frac{3k-3}{4} \right\rceil$. Furthermore, suppose that for all $n = 2, \hdots, k-1$, we have $\gamma_s(P_2 \Box P_n) = \left\lceil \frac{3n+1}{4} \right\rceil$. Then, $m_2 = 0$.\label{lem-P2Pnm20}\end{lem}

\begin{proof}Suppose that $m_2 > 0$, then there is a type 2 copy of $P_2$. Then since $m_0 > 2$, this copy of $P_2$ must either be contained within a block or an end-block. Suppose first that it is contained within an end-block, of length $l$. Then the end-block must contain at least $l$ vertices in $S$. We can delete this end-block, obtaining $P_2 \Box P_{k-l}$. Then, from Corollary \ref{cor-trim} we have $|S| \geq \left\lceil \frac{3k-3l+1}{4} \right\rceil + l$ which is a contradiction. Hence, the type 2 copy must be contained within a block.

Suppose that the block containing the type 2 copy has length $l$. By the definition of a block, $l \geq 3$. Then the block contains at least $l-1$ vertices in $S$. We can delete this block, obtaining a union of two graphs $P_2 \Box P_{a}$ and $P_2 \Box P_{b}$, where $a+b=k-l$. Then, from Corollary \ref{cor-trim} we have $|S| \geq \left\lceil \frac{3a+1}{4} \right\rceil + \left\lceil \frac{3b+1}{4} \right\rceil + l-1 \geq \left\lceil \frac{3a+3b+1}{4} \right\rceil + l-1 = \left\lceil \frac{3k+l-3}{4} \right\rceil$. This is a contradiction unless $l = 3$. Hence, the block must have the pattern 020.

Now, suppose that all copies contained within this block are internal copies. Then the pattern 020 contradicts Lemma \ref{lem-P2Pn0mod4}, as the two type 0 copies are distance two apart. Hence, the block must not contain all internal copies. Also, the block cannot start at $P_2^2$, or else $P_2^4$ would be of type 0 which contradicts Lemma \ref{lem-P2Pn0mod4}. Similarly, since $k = 0\mbox{ mod }4$ by Lemma \ref{lem-P2Pn0mod4}, the block cannot start at $P_2^{k-3}$. Hence, there are only two possible options. Either the block starts at $P_2^1$, or $P_2^{k-2}$. Suppose it starts at $P_2^1$, and consider $P_2^2, P_2^3, P_2^4, P_2^5$. If there is an attack at $v^1_1$, the guard at $v^2_1$ must move to assist, which implies that $v^4_1 \in S$ in order to avoid leaving $v^3_1$ unguarded. An analagous argument can be made for an attack at $v^1_2$, implying that $v^4_2 \in S$. Hence, $P_2^4$ must be of type 2 as well. Then, since all type 2 copies must exist in a pattern 020, this implies that $P_2^5$ is of type 0, which contradicts Lemma \ref{lem-P2Pn0mod4}. This situation is displayed in Figure \ref{fig-P2Pnm20}. By symmetry, because $k = 0\mbox{ mod }4$, the same argument can be made for a block starting at $P_2^{k-2}$. Hence the pattern 020 cannot exist in $S$, and so $m_2 = 0.$\end{proof}

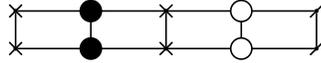
\begin{figure}[h!]\begin{center}\begin{tikzpicture}[smallgraph]\gridgraph{2}{5} \inS{2}{1} \inS{2}{2} \inSa{4}{1} \inSa{4}{2} \notS{1}{1} \notS{1}{2} \notS{3}{1} \notS{3}{2} \notSa{5}{1} \notSa{5}{2}\end{tikzpicture}
\caption{The situation described in Lemma \ref{lem-P2Pnm20}.\label{fig-P2Pnm20}}\end{center}\end{figure}

We are now ready to pose the main theorem of this section.

\begin{thm}Consider the graph $P_2 \Box P_n$. Then for $n \geq 2$,\label{thm-P2xPn}
$$\gamma_s(P_2 \Box P_n) = \left\lceil \frac{3n+1}{4} \right\rceil.$$\end{thm}
\begin{proof}We use the exact formulation from \cite{burdett} to confirm the result for $n \leq 16$. Then, suppose there is some value $k \geq 17$ such that Theorem \ref{thm-P2xPn} holds for $n = 2, \hdots, k-1$, but $\gamma_s(P_2 \Box P_k) \leq \left\lceil \frac{3k-3}{4} \right\rceil$. Then, there is a secure dominating set $S$ such that $|S| = \left\lceil \frac{3k-3}{4} \right\rceil$. By Lemma \ref{lem-P2Pn0mod4}, it must be the case that $k = 0\mbox{ mod }4$, and all internal copies $P_2^m$ of type 0 must satisfy $m \in \{2,3\}\mbox{ mod }4$. Furthermore, from Lemma \ref{lem-P2Pnm20}, $m_2 = 0$, and so from Lemma \ref{fig-P2PnG2}, $P_2^2$ and $P_2^{k-1}$ cannot be of type 0. It is also easy to check that neither $P_2^1$ nor $P_2^k$ can be of type 0, or else $P_2^2$ or $P_2^{k-1}$ need to be type 2 to satisfy the domination condition.

Given the above, we now consider all possible copies which might be type 0. They include $P_2^3$, $P_2^{k-2}$, and all internal pairs of the form $P_2^m, P_2^{m+1}$ for $m = 2\mbox{ mod }4$. There are $\frac{k}{4} - 2$ such internal pairs. Now, suppose that for one of these internal pairs, both are of type 0. Then it is clear from the domination condition that their neighbouring copies must be of type 2, which contradicts $m_2 = 0$. Hence, each internal pair contains at most one type 0 copy. However, since $k = 0\mbox{ mod }4$ and $m_2 = 0$, (\ref{eq-P2Pnm0}) reduces to $m_0 = \frac{k}{4}$. Hence we must have exactly one type 0 copy in all internal pairs, and $P_2^3$ and $P_2^{k-2}$ must both be type 0 as well.

Now, consider $P_2^2, \hdots, P_2^7$. From the previous paragraph, we know that exactly one of $P_2^6$ and $P_2^7$ must be type 0. If $P_2^6$ is type 0, then $P_2^7$ must be type 1, and this results in the pattern 101101, contradicting Lemma \ref{lem-P2Pn101101}. Hence, $P_2^6$ must be type 1 and $P_2^7$ type 0. The same argument can be made for $P_2^6, \hdots, P_2^{11}$, concluding that $P_2^{10}$ must be type 1 and $P_2^{11}$ type 0. Continuing this argument, we arrive at $P_2^{k-6}$ being type 0 and $P_2^{k-5}$ being type 1, but $P_2^{k-2}$ is also type 1. That means $P_2^{k-6}, \hdots, P_2^{k-1}$ has the pattern 101101, contradicting Lemma \ref{lem-P2Pn101101}. Hence, in all cases, a contradiction is reached, and so it cannot be the case that $|S| \leq \left\lceil \frac{3k-3}{4} \right\rceil$. Hence, the lower bound for $|S|$ coincides with the upper bound from Section \ref{sec-upperbounds}, completing the proof.\end{proof}

\section{The secure domination number of $P_2 \Box C_n$}

We begin by noting that $P_2 \Box C_n$ can be thought of as $P_2 \Box P_n$ with the addition of two edges. Hence $\gamma_s(P_2 \Box C_n) \leq \gamma_s(P_2 \Box P_n)$. The question then is, is the inequality ever strict, and if so, under which conditions? In this section, we will prove that equality holds in all cases except when $n = 0 \mbox{ mod}8$. Specifically, we will prove that
$$\gamma_s(P_2 \Box C_n) = \left\{\begin{array}{lcc}\frac{3n}{4}& & \mbox{ if } n = 0 \mbox{ mod}8,\\ \vspace{-0.2cm}\\
\left\lceil \frac{3n+1}{4} \right\rceil & & \mbox{otherwise.}\\
\end{array}\right.$$

Suppose that there is some value $k$ such that $\gamma_s(P_2 \Box C_k) < \gamma_s(P_2 \Box P_k)$. That is, there exists some secure dominating set $S$ for $P_2 \Box C_k$, such that $|S| = \left\lceil \frac{3k-3}{4} \right\rceil$. Note that this implies that $S$ is not a secure dominating set for $P_2 \Box P_k$. Recall that $P_2 \Box C_k$ contains $k$ copies of $P_2$, and each copy contains 0, 1 or 2 vertices in $S$. Note that $P_2 \Box C_k$ can be drawn as two concentric cycles connected by inner edges; we will arbitrarily designate one of these cycles to be the ``outer" cycle and the other to be the ``inner" cycle. To distinguish between these, we will refrain from using the term ``type 1" here, and instead use ``type 1i" and ``type 1o" to indicate a copy of $P_2$ has only the inner vertex in $S$, or only the outer vertex in $S$, respectively. Hence, we will say that each copy of $P_2$ will be one of four types: type 0, type 1i, type 1o or type 2.

\begin{lem}Suppose $S$ is a secure dominating set for $P_2 \Box C_k$ such that $|S| = \left\lceil \frac{3k-3}{4} \right\rceil.$ Then, no two neighbouring copies of $P_2$ may be of the same type.\end{lem}

\begin{proof}Suppose two copies are of the same type. Then, the two edges between them could be deleted, and by Lemmas \ref{thm-mods} and \ref{thm-mods2}, $S$ would still be secure dominating in the resulting graph. However, the resulting graph is $P_2 \Box P_k$, which leads to a contradiction.\end{proof}

\begin{lem}Suppose $S$ is a secure dominating set for $P_2 \Box C_k$ such that $|S| = \left\lceil \frac{3k-3}{4} \right\rceil.$ Further, suppose there is a block (or a set of consecutive blocks) of length $l$ which contains $c$ vertices in $S$. Then, $c < \frac{3l}{4}$, and if $c = \frac{3l-1}{4}$ then $k = 0\mbox{ mod }4$.\label{lem-P2Cncl}\end{lem}

\begin{proof}We can delete the block (or set of consecutive blocks) obtaining $P_2 \Box P_{k-l}$. Then, from Corollary \ref{cor-trim} we have $|S| \geq \left\lceil \frac{3k-3l+1}{4} \right\rceil + c = \left\lceil \frac{3k+4c-3l+1}{4} \right\rceil$. It is clear that this contradicts the assumption if $4c-3l+1 \geq 1$, and hence $c < \frac{3l}{4}$. Also, if $c = \frac{3l-1}{4}$ we have $\left\lceil \frac{3k}{4} \right\rceil \leq \left\lceil \frac{3k-3}{4} \right\rceil$ which implies that $k = 0\mbox{ mod }4$.\end{proof}

\begin{lem}Suppose $S$ is a secure dominating set for $P_2 \Box C_k$ such that $|S| = \left\lceil \frac{3k-3}{4} \right\rceil.$ If any copy of $P_2$ is of type 2, then its neighbours must both be type 0, and it must be the case that $k = 0\mbox{ mod }4$.\label{lem-P2Cn020}\end{lem}

\begin{proof}Consider such a type 2 copy of $P_2$. It must exist inside a block with length $l \geq 3$. Then it must contain at least $l-1$ vertices in $S$. Applying Lemma \ref{lem-P2Cncl}, we obtain a contradiction unless $l = 3$. Hence, any block containing the type 2 copy must have the pattern 020. Then, in this case $l = 3$ and $c = 2$, and so $c = \frac{3l-1}{4}$. Hence, from Lemma \ref{lem-P2Cncl}, $k = 0\mbox{ mod }4.$\end{proof}

\begin{lem}Suppose $S$ is a secure dominating set for $P_2 \Box C_k$ such that $|S| = \left\lceil \frac{3k-3}{4} \right\rceil.$ Then $S$ cannot contain the pattern 202, or the pattern 010.\label{lem-P2Cnpatterns}\end{lem}

\begin{proof}Suppose that $S$ contains the pattern 202. Then, from Lemma \ref{lem-P2Cn020}, both type 2 copies must have an additional type 0 neighbour, and so $S$ contains the pattern 02020. However, if we apply Lemma \ref{lem-P2Cncl} for $l = 5$ and $c = 4$, we obtain a contradiction. Hence, $S$ cannot contain the pattern 202.

Next, suppose that $S$ contains the pattern 010. We will denote these three copies by $G^1, G^2, G^3$, and also consider the next three neighbouring copies $G^4, G^5, G^6$. We will denote the outer and inner vertices of $G^i$ by $v^i_1$ and $v^i_2$, respectively. By symmetry, it doesn't matter whether $v^2_1 \in S$ or $v^2_2 \in S$; we will assume the former. Then, by the domination property, $v^4_2 \in S$. Also, if there is an attack at $v^2_2$, the guard at $v^2_1$ must move to assist, which implies that $v^4_1 \in S$ in order to avoid leaving $v^3_1$ unguarded. Hence, $G^4$ is of type 2, and by Lemma \ref{lem-P2Cn020} we know that $G^5$ is of type 0.

Then, if there is an attack at $v^3_2$, the guard at $v^4_2$ must move to assist, which implies that $v^6_2 \in S$ in order to avoid leaving $v^5_2$ unguarded. Likewise, if there is an attack at $v^3_1$, the guard at $v^2_1$ cannot move to assist without leaving $v^2_2$ unguarded. Hence, the guard at $v^4_1$ must move to assist, which implies that $v^6_1 \in S$ in order to avoid leaving $v^5_1$ unguarded. Therefore, $G^6$ is of type 2, and $S$ contains the pattern 202, which as established above is impossible. Hence, $S$ cannot contain the pattern 010. This situation is displayed in Figure \ref{fig-P2Cn010}\end{proof}

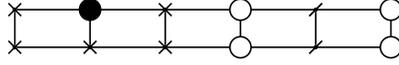
\begin{figure}[h!]\begin{center}\begin{tikzpicture}[smallgraph]\gridgraph{2}{6} \inS{2}{1} \inSa{4}{1} \inSa{4}{2} \inSa{6}{1} \inSa{6}{2}
\notS{1}{1} \notS{1}{2} \notS{2}{2} \notS{3}{1} \notS{3}{2} \notSa{5}{1} \notSa{5}{2}\end{tikzpicture}
\caption{The situation described in the second part of Lemma \ref{lem-P2Cnpatterns}.\label{fig-P2Cn010}}\end{center}\end{figure}

\begin{lem}Suppose $S$ is a secure dominating set for $P_2 \Box C_k$ such that $|S| = \left\lceil \frac{3k-3}{4} \right\rceil.$ Then $k = 0\mbox{ mod }4$.\label{lem-P2Cn0mod4}\end{lem}

\begin{proof}Suppose $k \neq 0\mbox{ mod }4$. We know from Lemma \ref{lem-P2Cn020} that no copies of $P_2$ are of type 2. Hence, all blocks have the pattern 01...10 for some number of 1s. We know from Lemma \ref{lem-P2Cnpatterns} that $S$ cannot contain the pattern 010. We will now show that $S$ cannot contain the pattern 0110 either in this case.

Consider six consecutive copies of $P_2$, denoted by $G^1, \hdots, G^6$. Suppose that $G^2$ and $G^5$ are of type 0, one of $G^3$ and $G^4$ is of type 1i, and the other of type 1o. Due to symmetry, it doesn't matter which, so we will assume that $v^3_1 \in S$ and $v^4_2 \in S$. This situation is displayed in Figure \ref{fig-P2Cn0110}. Then, by the domination property $v^1_2 \in S$ and $v^6_1 \in S$. Also, since there are no copies of type 2, we have $v^1_1 \not\in S$ and $v^6_2 \not\in S$. Then, suppose there is an attack at $v^3_2$. Either the guard at $v^3_1$, or the guard at $v^4_2$ must move to assist. In the former case, $v^2_1$ is left unguarded. In the latter case, $v^5_2$ is left unguarded. Hence, this is impossible, and $S$ cannot contain the pattern 0110. Therefore, all blocks contain at least three blocks of type 1i or type 1o.

Noting that $m_2 = 0$, we have $m_1 = \left\lceil \frac{3k-3}{4} \right\rceil$. Also, from the above paragraph, we have $m_1 \geq 3m_0$. Combining these together, we get $\frac{3k}{4} \leq \left\lceil \frac{3k-3}{4} \right\rceil$ which violates the initial assumption that $k \neq 0\mbox{ mod }4.$\end{proof}

\begin{figure}[h!]\begin{center}\begin{tikzpicture}[smallgraph]\gridgraph{2}{6} \inSa{1}{2} \inS{3}{1} \inS{4}{2} \inSa{6}{1}
\notSa{1}{1} \notS{2}{1} \notS{2}{2} \notS{3}{2} \notS{4}{1} \notS{5}{1} \notS{5}{2} \notSa{6}{2}\end{tikzpicture}
\caption{The situation described in Lemma \ref{lem-P2Cn0mod4}.\label{fig-P2Cn0110}}\end{center}\end{figure}
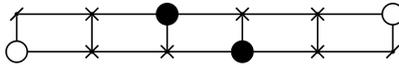

We have now established that if $\gamma_s(P_2 \Box C_k) < \gamma_s(P_2 \Box P_k)$, it must be the case that $k = 0\mbox{ mod }4$. Note that this, in turn, implies that $|S| = \frac{3k}{4}$. We will now go a step further and prove that, in fact, it must be the case that $k = 0\mbox{ mod }8$. Consider any block in $S$. From the results established so far, it is clear that the internal copies in the block are either alternating copies of type 1i and type 1o, or there is a single internal copy of type 2. We will refer to the former as a {\em type 1 block} of a certain length, and the latter as a {\em type 2 block}.

\begin{lem}Suppose $S$ is a secure dominating set for $P_2 \Box C_k$ such that $|S| = \frac{3k}{4}$. Suppose $S$ contains a type 1 block of length 4. Then there must be a type 2 block on at least one side of this block.\label{lem-P2Cnblocks}\end{lem}

\begin{proof}Suppose the opposite is true. Then $S$ contains the pattern 101101. However, the proof of Lemma \ref{lem-P2Pn101101} can be repeated in this situation too, leading to the conclusion that this is impossible.\end{proof}

Now, denote by $b_i$ the number of type 1 blocks of length $i$. Then the following result emerges immediately from Lemma \ref{lem-P2Cnblocks}, where equality occurs if every type 2 block has type 1 blocks of length 4 on both sides.

\begin{cor}Suppose $S$ is a secure dominating set for $P_2 \Box C_k$ such that $|S| = \frac{3k}{4}$. Then $b_4 \leq 2m_2$.\label{cor-P2Cnb4}\end{cor}

\begin{lem}Suppose $S$ is a secure dominating set for $P_2 \Box C_k$ such that $|S| = \frac{3k}{4}$. Then all type 1 blocks must have length either 4 or 5, and $b_4 = 2m_2$.\label{lem-P2Cnblocks2}\end{lem}

\begin{proof}We know from Lemma \ref{lem-P2Cnpatterns} that the type 1 blocks cannot have length 3. By considering the number of vertices each block contributes to $S$, we have $2m_2 + \sum_{i=4}^k (i-2)b_i = \frac{3k}{4}$. Also, by considering the number of copies of $P_2$ each block contains, we have $2m_2 + \sum_{i=4}^k (i-1)b_i = k$. Combining these, we can obtain $b_4 = 2m_2 + \sum_{i=6}^k (i-5)b_i$. Then, the result emerges from Corollary \ref{cor-P2Cnb4} and the nonnegativity of $b_i$.\end{proof}

Lemma \ref{lem-P2Cnblocks} and Lemma \ref{lem-P2Cnblocks2} imply that $S$ is exclusively made up of two kinds of patterns; specifically, the patterns 01102011 and 0111. Note that each pattern ends with either a type 1i or type 1o copy. We will say the pattern has ``parity i" in the former case, and ``parity o" in the latter case. This pattern is then followed by another pattern which begins with a type 0 copy.

\begin{lem}Suppose $S$ is a secure dominating set for $P_2 \Box C_k$ such that $|S| = \frac{3k}{4}$. Then $b_5$ must be an even number.\label{lem-P2Cneven}\end{lem}

\begin{proof}Suppose $S$ contains the pattern 11020110 with parity i. Note that the pattern must be preceded by another pattern which ends with 10, and followed by another pattern which begins with a 1. We now consider eleven consecutive copies of $P_2$, call them $G^1, \hdots, G^{11}$ and denote the outer and inner vertices of $G^i$ by $v^i_1$ and $v^i_2$ respectively. They have types obeying the pattern 10110201101, and $G^3$ has type 1i. This situation is displayed in Figure \ref{fig-P2Cn10110201101}. By the domination condition, $G^1$ is type 1o. If there is an attack at $v^2_2$, the guard at $v^3_2$ must move to assist, which implies that $v^4_1 \in S$ in order to avoid leaving $v^3_1$ unguarded. Then, if there is an attack at $v^5_2$, the guard at $v^6_2$ must move to assist, which implies that $v^8_2 \in S$ in order to avoid leaving $v^7_2$ unguarded. Then, if there is an attack at $v^7_2$, the guard at $v^6_2$ cannot move to assist without leaving $v^5_2$ unguarded. Therefore, the guard at $v^8_2$ must move to assist, which implies that $v^9_1 \in S$ in order to avoid leaving $v^8_1$ unguarded. Finally, by the domination condition, $G^{11}$ is type 1i. As a result, we see that if the pattern 11020110 has parity i, the next pattern also has parity i. The analagous result can be obtained if the pattern 11020110 has parity o.

Now, suppose $S$ contains the pattern 1110 with parity i. Then clearly the three type 1 copies in the pattern are of type 1i, 1o, and 1i respectively. Then by the domination condition, the next pattern must start with a type 1o copy. Hence, if the pattern 1110 has parity i, the next pattern has parity o. The analogous result can be obtained if the pattern 1110 has parity o.

Since the parity is not changed by 11020110, and is changed by 1110, there must be an even number of 1110 patterns.\end{proof}

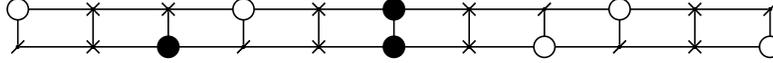
\begin{figure}[h!]\begin{center}\begin{tikzpicture}[smallgraph]\gridgraph{2}{11} \inSa{1}{1} \inS{3}{2} \inSa{4}{1} \inS{6}{1} \inS{6}{2} \inSa{8}{2} \inSa{9}{1} \inSa{11}{2}
\notSa{1}{2} \notS{2}{1} \notS{2}{2} \notS{3}{1} \notSa{4}{2} \notS{5}{1} \notS{5}{2} \notS{7}{1} \notS{7}{2} \notSa{8}{1} \notSa{9}{2} \notS{10}{1} \notS{10}{2} \notSa{11}{1}\end{tikzpicture}
\caption{The situation described in the second part of Lemma \ref{lem-P2Cneven}.\label{fig-P2Cn10110201101}}\end{center}\end{figure}

We are now ready to prove the main theorem of this section.

\begin{thm}Consider $P_2 \Box C_n$ for $n \geq 3$. Then,\label{thm-P2Cn}
$$\gamma_s(P_2 \Box C_n) = \left\{\begin{array}{ccc}\frac{3n}{4} & & n = 0\mbox{ mod }8,\\
\left\lceil \frac{3n+1}{4} \right\rceil & & \mbox{otherwise.}\end{array}\right.$$\end{thm}

\begin{proof}From Lemma \ref{lem-P2Cneven}, we know that if $n \neq 0\mbox{ mod }8$ then $\gamma_s(P_2 \Box C_n) = \gamma_s(P_2 \Box P_n)$. Also, from Section \ref{sec-upperbounds} we know that if $n = 0\mbox{ mod }8$, we have $\gamma_s(P_2 \Box C_n) \leq \frac{3n}{4}$. Then all that remains is to show that $\gamma_s(P_2 \Box C_n) \geq \frac{3n}{4}$ for $n = 0\mbox{ mod }8$.

Suppose there is a secure dominating set $S$ for $P_2 \Box C_n$, $n = 0\mbox{ mod }8$, such that $|S| \leq \frac{3n}{4} - 1$.Then, certainly there is at least one type 0 block. It can be deleted to obtain $P_2 \Box P_{n-1}$. Then, from Corollary \ref{cor-trim} we have $|S| \geq \left\lceil \frac{3n-2}{4} \right\rceil$. Since $n = 0\mbox{ mod }4$, we know that $\left\lceil \frac{3n-2}{4} \right\rceil = \frac{3n}{4}$, which contradicts the initial assumption.\end{proof}

We conclude this section by remarking that the secure domination number for $P_2 \Box C_n$ has been conjectured twice previously, in different contexts, both of which agree with Theorem \ref{thm-P3Cn}. Valveny and Rodriguez-Vel\'{a}zquez \cite{valveny} considered $P_2 \Box C_n$ directly, while Winter \cite{winter} gave a conjecture for the secure domination number of the generalized Petersen graph $P(n,1)$ which is isomorphic to $P_2 \Box C_n$. Interestingly, the latter gave the conjectured formula as $\left\lceil \frac{n+7}{8} \right \rceil + \left\lceil \frac{n+4}{8} \right \rceil + \left\lceil \frac{n+2}{4} \right \rceil + \left\lceil \frac{n+1}{4} \right\rceil$. By checking all choices of $n\mbox{ mod }8$, it can be seen that this coincides with Theorem \ref{thm-P2Cn}.

\subsection{The secure domination number of $M_{2n}$}

Much like $P_2 \Box C_n$, the M\"{o}bius ladder graph $M_{2n}$ can be thought of as $P_2 \Box P_n$ with the addition of two edges. In this case, the two edges form a ``twist", however the location of this twist is not fixed. Indeed, we can choose to think of the twist as occurring between any two neighbouring copies of $P_2$. As such, whenever we look at proofs which are contained within a subgraph of $M_{2n}$, we can choose to think of the twist as occurring elsewhere. It can be quickly checked that all of the proofs for $P_2 \Box C_n$ prior to Lemma \ref{lem-P2Cneven} are also applicable to $M_{2n}$. In particular, it is hence established that if there is a secure dominating set $S$ for $M_{2k}$ such that $|S| < \gamma_S(P_2 \Box P_k)$, then $S$ must be made up exclusively of two kinds of patterns; specifically, the patterns 01102011 and 0111. Then, since we can think of the twist occuring between any neighbouring copies of $P_2$, we imagine it occuring between two patterns. We can then follow the same argument as in Lemma \ref{lem-P2Cneven}, noting the twist, to obtain the following corollary.

\begin{cor}Suppose $S$ is a secure dominating set for $M_{2k}$ such that $|S| = \frac{3k}{4}$. Then $b_5$ must be an odd number.\label{cor-M2nodd}\end{cor}

Finally, we can use analagous arguments as in Theorem \ref{thm-P2Cn}, substituting Corollary \ref{cor-M2nodd} for Lemma \ref{lem-P2Cneven} and using the appropriate upper bound from Section \ref{sec-upperbounds} to obtain the following result. The only case not covered by the analagous arguments is $n = 4$, however we have confirmed that $\gamma_s(M_8) = 3$ using the exact formulation from \cite{burdett}.

\begin{thm}Consider $M_{2n}$ for $n \geq 3$. Then,\label{thm-M2n}
$$\gamma_s(M_{2n}) = \left\{\begin{array}{ccc}\frac{3n}{4} & & n = 4\mbox{ mod }8,\\
\left\lceil \frac{3n+1}{4} \right\rceil & & \mbox{otherwise.}\end{array}\right.$$\end{thm}

\section{The secure domination number of $P_3 \Box P_n$}\label{secp3pn}

We now consider $P_3 \Box P_n$. From Cockayne et al \cite{cockayne} we know that $\gamma_s(P_3 \Box P_n) \leq n+2$ for all $n \geq 2$. In this section, we will prove the following for $n\geq2$,

$$\gamma_s(P_3 \Box P_n) = \left\{\begin{array}{ccc}n+1 & & n \leq 8\mbox{ or } n = 10,\\
n + 2 & & n = 9\mbox{ or } n \geq 11.\end{array}\right.$$

We use the exact formulation from \cite{burdett} to confirm the result for $n \leq 21$. Then, suppose there is some value $k \geq 22$ such that $\gamma_s(P_3 \Box P_n) = n+2$ for $n = 11, \hdots, k-1$ but $\gamma_s(P_3 \Box P_k) \leq k+1$. Then, there is a secure dominating set $S$ for $P_3 \Box P_k$ such that $|S| = k+1$.

Recall that $P_3 \Box P_k$ contains $k$ copies of $P_3$, which we call $P_3^1, \hdots, P_3^k$. We will refer to $P_3^1$ and $P_3^k$ as ``end copies", $P_3^2$ and $P_3^{k-1}$ as ``second-end copies", and all other copies as ``internal copies".

\begin{lem}If $S$ is a secure dominating set for $P_3 \Box P_k$ such that $|S| = k+1$, then $m_0 > 0$.\label{lem-P3Pnm0}\end{lem}

\begin{proof}From equation (\ref{eq-balance}) we have $m_0 = -1 + m_2 + 2m_3$. Now, suppose that $m_0 = 0$. Then it is clear that $m_3 = 0$, $m_2 = 1$ and $m_1 = k-1$. Now suppose that none of $P_3^1, P_3^2, P_3^3, P_3^4$ are type 2. Clearly $P_3^1$ cannot be type 0, or else $P_3^2$ would need to be type 3 to satisfy the domination condition. Hence, $P_3^1$ is type 1. Up to symmetry, there are two cases to consider. Either one of the upper or lower vertices is in $S$, or the middle vertex is in $S$.

For the former case, we will assume $v^1_1 \in S$. Then, by the domination condition, $v^2_3 \in S$. Now, if there is an attack at $v^1_3$, the guard at $v^2_3$ must move to assist, which implies that $v^3_2 \in S$ in order to avoid leaving $v^2_2$ unguarded. Then, if there is an attack at $v^1_2$, the guard at $v^1_1$ must move to assist, leaving $v^2_1$ unguarded. Hence, this case is impossible.

For the second case, we have $v^1_2 \in S$. Then, suppose $v^2_1 \not\in S$. If there is an attack at $v^1_1$, the guard at $v^1_2$ must move to assist, which implies that $v^2_3 \in S$ in order to avoid leaving $v^1_3$ unguarded. It is clear that either the upper or lower vertex of $P^3_2$ must be in $S$, and due to symmetry either choice is equivalent. We will assume $v^2_3 \in S$. Then, by the domination condition, $v^3_1 \in S$. Then, if there is an attack at $v^2_1$, the guard at $v^3_1$ must move to assist, which implies that $v^4_2 \in S$ in order to avoid leaving $v^3_2$ unguarded. Finally, if there is an attack at $v^1_3$, either the guard at $v^1_2$, or the guard at $v^2_3$, must move to assist. In the former case, $v^1_1$ is left unguarded. In the latter case, $v^3_3$ is left unguarded. Hence, this case is also impossible, and so it cannot be the case that none of $P_3^1, P_3^2, P_3^3, P_3^4$ are type 2. These two cases are displayed in Figure \ref{fig-P3Pncases}.

An analagous argument can be made for $P^3_{k-3}, P^3_{k-2}, P^3_{k-1}, P^3_k$, and so $m_2 > 2$, which is a contradiction, completing the proof.\end{proof}

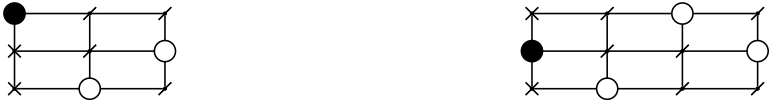
\begin{figure}[h!]\begin{center}
\begin{minipage}{0.4\textwidth}\begin{tikzpicture}[smallgraph]
\gridgraph{3}{3}
\inS{1}{1} \inSa{2}{3} \inSa{3}{2}
\notS{1}{2} \notS{1}{3} \notSa{2}{1} \notSa{2}{2} \notSa{3}{1} \notSa{3}{3}\end{tikzpicture}
\end{minipage}\hspace*{1cm}\begin{minipage}{0.4\textwidth}\begin{tikzpicture}[smallgraph]
\gridgraph{3}{4}
\inS{1}{2} \inSa{2}{3} \inSa{3}{1} \inSa{4}{2}
\notS{1}{1} \notS{1}{3} \notSa{2}{1} \notSa{2}{2} \notSa{3}{2} \notSa{3}{3} \notSa{4}{1} \notSa{4}{3}
\end{tikzpicture}\end{minipage}
\caption{The two cases considered in Lemma \ref{lem-P3Pnm0}.\label{fig-P3Pncases}}\end{center}\end{figure}

We are now ready to prove the main result of this section.

\begin{thm}Consider $P_3 \Box P_n$ for $n \geq 2$. Then,\label{thm-P3Pn}
$$\gamma_s(P_3 \Box P_n) = \left\{\begin{array}{ccc}n+1 & & n \leq 8\mbox{ or } n = 10,\\
n + 2 & & n = 9\mbox{ or } n \geq 11.\end{array}\right.$$\end{thm}

\begin{proof}We use the exact solver to verify Theorem \ref{thm-P3Pn} for $n \leq 21$. From Section \ref{sec-upperbounds} we know that $\gamma_s(P_3 \Box P_n) \leq n+2$. Then, it suffices to show that for any $P_3 \Box P_k$, $k \geq 22$, it is impossible to have a secure dominating set $S$ such that $|S| = k+1$. Suppose that such an $S$ exists. Then for some value of $k$, $P_3 \Box P_k$ constitutes the minimal such example. From Lemma \ref{lem-P3Pnm0} we know that $m_0 > 0$.

Suppose that there is an internal copy of type 0, say $P_3^m$ for $m = 3, \hdots, k-2$. Then we can delete this copy, resulting in the union of $P_3 \Box P_{m-1}$ and $P_3 \Box P_{k-m}$. Since $k \geq 22$, at least one of $m-1$ and $k-m$ is either 9, or at least 11. Then, from Corollary \ref{cor-trim} we have $|S| \geq k + 2$, which contradicts our initial assumption. Therefore, there are no internal copies of type 0.

Suppose that $P_3^2$ is of type 0. In order to satisfy the domination condition, $P_3^1$ must have at least one vertex in $S$. Suppose $P_3^1$ is type 1, then $v^1_2 \in S$. However, if there is an attack at $v^1_1$, then the guard at $v^1_2$ must move to assist, leaving $v^1_3$ unguarded. Hence, $P_3^1$ must contain at least two vertices in $S$. Then, we can delete $P_3^1, P_3^2$, to obtain $P_3 \Box P_{k-2}$. Then, from Corollary \ref{cor-trim} we have $|S| \geq k+2$, which contradicts our initial assumption. An analagous argument can be made for $P_3^{k-1}$, and therefore there are no second-end copies of type 0.

At this stage, we know that either one or both of $P_3^1$ and $P_3^k$ must be type 0. Suppose they both are. Then, in order to satisfy the domination condition, $P_3^2$ and $P_3^{k-1}$ must both be type 3, which violates equation (\ref{eq-balance}). Hence, exactly one of them must be type 0. Due to symmetry, the choice is equivalent, so we assume $P_3^1$ is type 0. Then, $P_3^2$ is type 3, and all remaining copies must be type 1. However, this implies that $P_3^{k-3}, P_3^{k-2}, P_3^{k-1}, P_3^k$ are all type 1, and using the same argument as used in the proof of Lemma \ref{lem-P3Pnm0}, this is impossible. Hence, the initial assumption must be incorrect, completing the proof.\end{proof}

\section{The secure domination number of $P_3 \Box C_n$}

We now consider $P_3 \Box C_n$. We will prove that:

$$\gamma_s(P_3 \Box C_n) = \left\{\begin{array}{lcc}n + 1 & & \mbox{ if } n = 4, 7,\\
3 \left\lceil \frac{n}{3} \right\rceil & & \mbox{otherwise.}\\
\end{array}\right.$$

We use the exact formulation from \cite{burdett} to confirm the above for $n \leq 26$. Then, suppose there is a value $k \geq 27$ such that the above is true for $n = 3, \hdots, k-1$, but $\gamma_s(P_3 \Box C_k) = 3 \left\lceil \frac{n}{3} \right\rceil - 1$. We will first focus on the cases when $k = 1\mbox{ mod }3$, and then handle the other cases afterwards. That is, there is a secure dominating set $S$ for $P_3 \Box C_k$ such that $|S| = k+1$.

\begin{thm}Suppose that $S$ is a secure dominating set for $P_3 \Box C_n$, and that $S$ contains the pattern 111111. Whichever vertex from the first copy is in $S$, the same vertex from the fourth copy is also in $S$.\label{thm-111111}\end{thm}

\begin{proof}Suppose it is not true. In Figure \ref{fig-9cases} we display, up to symmetry, all possible situations in which the first and fourth copies have different vertices in $S$, and no neighbouring copies share the same vertex in $S$. In cases 1, 2, 7 and 9, the domination property is not satisfied. For the remaining cases, we will demonstrate that an attack at a certain vertex forces a guard to move and leave another vertex unguarded. In each case, we refer to the six copies as $G^1, \hdots, G^6$ respectively, and the choice of vertices to be contained in $S$ is fixed for $G^1, \hdots, G^4$. We will refer to the top, central, and bottom vertices of $G^i$ as $v^i_1$, $v^i_2$ and $v^i_3$ respectively.

In case 3, if there is an attack at $v^2_2$, the guard at $v^2_3$ is forced to move to assist, leaving $v^3_3$ unguarded.

In case 4, if there is an attack at $v^3_2$, the guard at $v^2_2$ cannot move to assist without leaving $v^2_3$ unguarded. Hence, the guard at $v^3_1$ must move to assist, which implies that $v^5_1 \in S$ in order to avoid leaving $v^4_1$ unguarded. Then, if there is an attack at $v^4_2$, the guard at $v^4_3$ is forced to move to assist, leaving $v^3_3$ unguarded.

In case 5, in order to satisfy the domination condition, there must be a guard at $v^5_1$. Then, if there is an attack at $v^4_1$, the guard at $v^5_1$ must move to assist, which implies that $v^6_2 \in S$ in order to avoid leaving $v^5_2$ unguarded. Then, if there is an attack at $v^4_2$, either the guard at $v^3_2$ or the guard at $v^4_3$ must move to assist. In the former case, $v^3_1$ is left unguarded. In the latter case, $v^5_3$ is left unguarded.

In case 6, if there is an attack at $v^3_2$, the guard at $v^3_1$ must move to assist, which implies that $v^5_1 \in S$ in order to avoid leaving $v^4_1$ unguarded. Then, if there is an attack at $v^3_3$, either the guard at $v^2_3$ or the guard at $v^4_3$ must move to assist. In the former case, $v^2_2$ is left unguarded. In the latter case, $v^4_2$ is left unguarded.

In case 8, if there is an attack at $v^3_2$, the guard at $v^3_3$ must move to assist, leaving $v^2_3$ unguarded.\end{proof}

\begin{figure}[h!]\begin{center}
\begin{minipage}{0.03\textwidth}(1)\end{minipage} \begin{minipage}{0.4\textwidth}\begin{tikzpicture}[smallgraph]
\gridgraph{3}{6}
\inS{1}{1} \inS{2}{2} \inS{3}{1} \inS{4}{2}
\notS{1}{2} \notS{1}{3} \notS{2}{1} \notS{2}{3} \notS{3}{2} \notS{3}{3} \notS{4}{1} \notS{4}{3}
\end{tikzpicture}
\end{minipage}\hspace*{1cm}\begin{minipage}{0.03\textwidth}(6)\end{minipage} \begin{minipage}{0.4\textwidth}\begin{tikzpicture}[smallgraph]
\gridgraph{3}{6}
\inS{1}{1} \inS{2}{3} \inS{3}{1} \inS{4}{3} \inSa{5}{1}
\notS{1}{2} \notS{1}{3} \notS{2}{1} \notS{2}{2} \notS{3}{2} \notS{3}{3} \notS{4}{1} \notS{4}{2} \notSa{5}{2} \notSa{5}{3}
\end{tikzpicture}\end{minipage}

\vspace*{0.25cm}\begin{minipage}{0.03\textwidth}(2)\end{minipage} \begin{minipage}{0.4\textwidth}\begin{tikzpicture}[smallgraph]
\gridgraph{3}{6}
\inS{1}{1} \inS{2}{2} \inS{3}{3} \inS{4}{2}
\notS{1}{2} \notS{1}{3} \notS{2}{1} \notS{2}{3} \notS{3}{1} \notS{3}{2} \notS{4}{1} \notS{4}{3}
\end{tikzpicture}
\end{minipage}\hspace*{1cm}\begin{minipage}{0.03\textwidth}(7)\end{minipage} \begin{minipage}{0.4\textwidth}
\begin{tikzpicture}[smallgraph]
\gridgraph{3}{6}
\inS{1}{2} \inS{2}{1} \inS{3}{2} \inS{4}{1}
\notS{1}{1} \notS{1}{3} \notS{2}{2} \notS{2}{3} \notS{3}{1} \notS{3}{3} \notS{4}{2} \notS{4}{3}
\end{tikzpicture}\end{minipage}

\vspace*{0.25cm}\begin{minipage}{0.03\textwidth}(3)\end{minipage} \begin{minipage}{0.4\textwidth}\begin{tikzpicture}[smallgraph]
\gridgraph{3}{6}
\inS{1}{1} \inS{2}{3} \inS{3}{1} \inS{4}{2}
\notS{1}{2} \notS{1}{3} \notS{2}{1} \notS{2}{2} \notS{3}{2} \notS{3}{3} \notS{4}{1} \notS{4}{3}
\end{tikzpicture}
\end{minipage}\hspace*{1cm}\begin{minipage}{0.03\textwidth}(8)\end{minipage} \begin{minipage}{0.4\textwidth}\begin{tikzpicture}[smallgraph]
\gridgraph{3}{6}
\inS{1}{2} \inS{2}{1} \inS{3}{3} \inS{4}{1}
\notS{1}{1} \notS{1}{3} \notS{2}{2} \notS{2}{3} \notS{3}{1} \notS{3}{2} \notS{4}{2} \notS{4}{3}
\end{tikzpicture}\end{minipage}

\vspace*{0.25cm}\begin{minipage}{0.03\textwidth}(4)\end{minipage} \begin{minipage}{0.4\textwidth}\begin{tikzpicture}[smallgraph]
\gridgraph{3}{6}
\inS{1}{1} \inS{2}{2} \inS{3}{1} \inS{4}{3} \inSa{5}{1}
\notS{1}{2} \notS{1}{3} \notS{2}{1} \notS{2}{3} \notS{3}{2} \notS{3}{3} \notS{4}{1} \notS{4}{2} \notSa{5}{2} \notSa{5}{3}
\end{tikzpicture}
\end{minipage}\hspace*{1cm}\begin{minipage}{0.03\textwidth}(9)\end{minipage} \begin{minipage}{0.4\textwidth}\begin{tikzpicture}[smallgraph]
\gridgraph{3}{6}
\inS{1}{2} \inS{2}{3} \inS{3}{2} \inS{4}{1}
\notS{1}{1} \notS{1}{3} \notS{2}{1} \notS{2}{2} \notS{3}{1} \notS{3}{3} \notS{4}{2} \notS{4}{3}
\end{tikzpicture}\end{minipage}

\vspace*{0.25cm}\begin{minipage}{0.03\textwidth}(5)\end{minipage} \begin{minipage}{0.4\textwidth}\begin{tikzpicture}[smallgraph]
\gridgraph{3}{6}
\inS{1}{1} \inS{2}{3} \inS{3}{2} \inS{4}{3} \inSa{5}{1} \inSa{6}{2}
\notS{1}{2} \notS{1}{3} \notS{2}{1} \notS{2}{2} \notS{3}{1} \notS{3}{3} \notS{4}{1} \notS{4}{2} \notSa{5}{2} \notSa{5}{3} \notSa{6}{1} \notSa{6}{3}
\end{tikzpicture}
\end{minipage}\hspace*{1cm}\begin{minipage}{0.03\textwidth}\hspace*{20cm}\end{minipage}\begin{minipage}{0.4\textwidth}\hspace*{20cm}\end{minipage}
\end{center}
\caption{9 cases\label{fig-9cases}}\end{figure}

The next corollary follows immediately from Theorem \ref{thm-111111} and the fact that neighbouring copies cannot have the same vertices in $S$. To see this, suppose that neighbouring copies $P^i,P^{i+1}$ have the same vertices in $S$. Then by Lemmas \ref{thm-mods} and \ref{thm-mods2}, we can delete all edges between $P^i$ and $P^{i+1}$ to obtain the graph $P_3 \Box P_k$ for which $S$ remains secure dominating. Thus $|S| \geq \gamma_s(P_3 \Box P_{k}) \geq n+1$, which is a contradiction.

\begin{cor}Suppose $S$ is a secure dominating set for $P_3 \Box C_k$, $k = 1\mbox{ mod }3$, such that $|S| = k+1$. If $S$ contains the pattern 1111, then the first three copies of $P_3$ in that pattern each contain a different vertex in $S$.\label{cor-P3Cnthreedifferent}\end{cor}

Corollary \ref{cor-P3Cnthreedifferent} itself leads to another corollary.

\begin{cor}Suppose $S$ is a secure dominating set for $P_3 \Box C_k$. If every copy of $P_3$ in $S$ is of type 1, then $k = 0\mbox{ mod }3$.\label{cor-P3Cn0mod3}\end{cor}

As discussed in Section \ref{sec-notation}, the proof of the following lemma is a more general version of that for Lemma \ref{lem-smoothing}. Instead of removing a set of copies, we instead identify a subgraph which contains some, but not all, vertices from various consecutive copies, and which only connects to the rest of the graph via edges whose endpoints are not in $S$. Then we are able to ``smooth" out this subgraph in a natural way. We will show here the full process of deleting the relevant edges, setting aside the disconnected subgraph, and adding in the new edges to complete the smoothing.

\begin{lem}Suppose $S$ is a secure dominating set for $P_3 \Box C_k$, $k = 1\mbox{ mod }3$, such that $|S| = k+1$, but that $\gamma_s(P_3 \Box C_{k-3}) = k-1$. Then $S$ does not contain the pattern 111111111.\label{lem-P3Cn1s}\end{lem}

\begin{proof}Suppose that $S$ does contain the pattern 111111111. From Corollary \ref{cor-P3Cnthreedifferent}, the first three copies in that pattern contain a different vertex in $S$, and then from Theorem \ref{thm-111111} this is then repeated for the next three copies, and then the three copies after that. It is clear then that this turns into a diagonal pattern which begins at one of the first three copies, and then continues throughout the pattern. The diagonal pattern can either go from the top vertex down to the bottom vertex, or vice versa; due to symmetry, these are equivalent, so we will assume the former.

Now, suppose we look at one instance of this diagonal pattern starting at either the third, fourth or fifth copy. An example for the diagonal pattern starting at the third copy is displayed in part 1 of Figure \ref{fig-P3Cndiagonal}. Then it is possible to identify a set of five edges to the left of the pattern which are all between vertices not in $S$, and likewise five edges to the right of the pattern, that can be deleted, and by Corollary \ref{cor-trim}, $S$ is still secure dominating in the resulting, disconnected graph. Finally, we can add in five edges to rejoin the pattern together, which again does not prevent $S$ from being secure dominating in the resulting graph. What results is the union of $P_3 \Box C_{k-3}$, plus another subgraph which contains exactly three vertices in $S$. Hence, $|S| \geq k-1+3$ which contradicts the initial assumption. Note that the process of deleting edges spanned seven copies of $P_3$, starting two copies to the left of the diagonal pattern, and ending four copies to the right. Then, since the diagonal pattern will start at either the third, fourth or fifth copy, this process can be performed if there are at least nine consecutive type 1 copies, and so $S$ cannot contain the pattern 111111111. The process of deleting and adding edges is illustrated in parts 2 and 3 of Figure \ref{fig-P3Cndiagonal}.\end{proof}

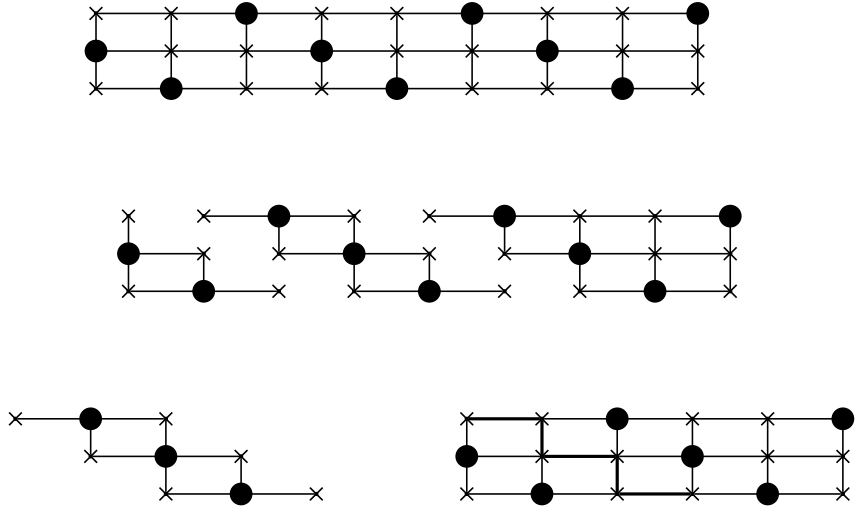
\begin{figure}[h!]\begin{center}\begin{tikzpicture}[smallgraph]\gridgraph{3}{9}
\inS{1}{2} \inS{2}{3} \inS{3}{1} \inS{4}{2} \inS{5}{3} \inS{6}{1} \inS{7}{2} \inS{8}{3} \inS{9}{1}
\notS{1}{1} \notS{1}{3} \notS{2}{1} \notS{2}{2} \notS{3}{2} \notS{3}{3} \notS{4}{1} \notS{4}{3} \notS{5}{1} \notS{5}{2}
\notS{6}{2} \notS{6}{3} \notS{7}{1} \notS{7}{3} \notS{8}{1} \notS{8}{2} \notS{9}{2} \notS{9}{3}
\end{tikzpicture}

\begin{tikzpicture}[smallgraph]
\foreach \n in {1,...,3} {\foreach \m in {1,...,9} {\node (\n\m) at (\m,-0.5*\n) {};}}
\draw (11) -- (31); \draw (21) -- (22); \draw (22) -- (32); \draw (31) -- (33); \draw (12) -- (14); \draw (13) -- (23);
\draw (23) -- (25); \draw (14) -- (34); \draw (34) -- (36); \draw (25) -- (35); \draw (15) -- (19); \draw (26) -- (29);
\draw (37) -- (39); \draw (16) -- (26); \draw (17) -- (37); \draw (18) -- (38); \draw (19) -- (39);
\inS{1}{2} \inS{2}{3} \inS{3}{1} \inS{4}{2} \inS{5}{3} \inS{6}{1} \inS{7}{2} \inS{8}{3} \inS{9}{1}
\notS{1}{1} \notS{1}{3} \notS{2}{1} \notS{2}{2} \notS{3}{2} \notS{3}{3} \notS{4}{1} \notS{4}{3} \notS{5}{1} \notS{5}{2}
\notS{6}{2} \notS{6}{3} \notS{7}{1} \notS{7}{3} \notS{8}{1} \notS{8}{2} \notS{9}{2} \notS{9}{3}
\draw[white] (0,0) -- (0,1);
\end{tikzpicture}

\begin{tikzpicture}[smallgraph]
\foreach \n in {1,...,3} {\foreach \m in {7,...,12} {\node (\n\m) at (\m,-0.5*\n) {};}}
\node (11) at (1,-0.5) {}; \node (12) at (2,-0.5) {}; \node (13) at (3,-0.5) {};
\node (22) at (2,-1) {}; \node (23) at (3,-1) {}; \node (24) at (4,-1) {};
\node (33) at (3,-1.5) {}; \node (34) at (4,-1.5) {}; \node (35) at (5,-1.5) {};
\draw (11) -- (13); \draw (12) -- (22); \draw (22) -- (24); \draw (13) -- (33); \draw (33) -- (35); \draw (24) -- (34);
\draw (27) -- (28); \draw (37) -- (39); \draw (18) -- (112); \draw (29) -- (212); \draw (310) -- (312); \draw (17) -- (37);
\draw (18) -- (38); \draw (19) -- (39); \draw (110) -- (310); \draw (111) -- (311); \draw (112) -- (312); \draw[very thick] (17) -- (18);
\draw[very thick] (18) -- (28); \draw[very thick] (28) -- (29); \draw[very thick] (29) -- (39); \draw[very thick] (39) -- (310);
\inS{2}{1} \inS{3}{2} \inS{4}{3} \inS{7}{2} \inS{8}{3} \inS{9}{1} \inS{10}{2} \inS{11}{3} \inS{12}{1}
\notS{1}{1} \notS{2}{2} \notS{3}{1} \notS{3}{3} \notS{4}{2} \notS{5}{3} \notS{7}{1} \notS{7}{3} \notS{8}{1} \notS{8}{2}
\notS{9}{2} \notS{9}{3} \notS{10}{1} \notS{10}{3} \notS{11}{1} \notS{11}{2} \notS{12}{2} \notS{12}{3}
\draw[white] (0,0) -- (0,1);
\end{tikzpicture}

\caption{The situation described in the second part of Lemma \ref{lem-P3Cn1s}. The first part shows an example of the pattern 111111111. The second part shows the pattern after the ten edges are deleted. The third part shows the final situation, with the disconnected subgraph moved aside, and five new edges added to reconnected the pattern.\label{fig-P3Cndiagonal}}\end{center}\end{figure}

\begin{lem}Suppose $S$ is a secure dominating set for $P_3 \Box C_k$, $k = 1\mbox{ mod }3$, such that $|S| = k+1$, but that $\gamma_s(P_3 \Box C_{k-3}) = k-1$. Then $m_3 = 0$.\label{lem-P3Cnm3}\end{lem}

\begin{proof}Recall that the assumption is known to be false for $k \leq 26$. Hence, $k \geq 27$. Then, suppose $m_3 > 0$. Consider any block containing a type 3 copy. Suppose the block has length $l$, then it contains at least $l$ vertices in $S$. Then, from Lemma \ref{lem-smoothing} we obtain $|S| \geq \gamma_s(P_3 \Box P_{k-l}) + l$. By the results of Section \ref{secp3pn}, this is a contradiction unless $k-l \in \{4,7\}$ and the block contains no type 2 copies and exactly one type 3 copy. However, from Lemma \ref{lem-P3Cn1s}, such a block could only be at most length 19. Since $k \geq 27$, this is a contradiction.

The only remaining possibility is that there are no blocks, which implies that $m_0 \leq 1$. If $m_0 = 0$ then from equation (\ref{eq-balance}) we have $m_3 = 0$, violating the initial assumption. Hence, $m_0 = 1$, $m_1 = k-2$, $m_2 = 0$, and $m_3 = 1$. From Lemma \ref{lem-P3Cn1s}, this implies that $k \leq 18$, which is a contradiction. Hence, $m_3 = 0$.\end{proof}

\begin{lem}Suppose $S$ is a secure dominating set for $P_3 \Box C_k$, $k = 1\mbox{ mod }3$, such that $|S| = k+1$, but that $\gamma_s(P_3 \Box C_{k-3}) = k-1$. Then $S$ does not contain the pattern 011.\label{lem-P3Cn011}\end{lem}

\begin{proof}Suppose that $S$ does contain the pattern 011. Consider four consecutive copies of $P_3$, denoted by $G^1, \hdots, G^4$, and suppose that $G^2$ is type 0, while $G^3$ and $G^4$ both are type 1. Then due to symmetry, there are two cases to consider. Either the middle vertex of $G^3$ is in $S$, or one of the upper and lower vertices of $G^3$ is in $S$.

For the former case, we have $v^3_2 \in S$. In order to satisfy the domination condition, we must have $v^1_1 \in S$ and $v^1_3 \in S$. Then, from Lemma \ref{lem-P3Cnm3} we know that $m_3 = 0$, and so $v^1_2 \not\in S$. Then, suppose there is an attack at $v^2_2$. The guard at $v^3_2$ must move to assist, which implies that both $v^4_1 \in S$ and $v^4_3 \in S$ in order to avoid leaving $v^3_1$ or $v^3_3$ unguarded.

For the latter case, due to symmetry, we can choose either $v^3_1 \in S$ or $v^3_3 \in S$; we will assume $v^3_1 \in S$. In order to satisfy the domination condition, we must have $v^1_2 \in S$, $v^1_3 \in S$, and $v^4_3 \in S$. Also, from Lemma \ref{lem-P3Cnm3} we know that $m_3 = 0$, and so $v^1_1 \not\in S$. Then, suppose there is an attack at $v^2_1$. The guard at $v^3_1$ must move to assist, which leaves $v^3_2$ unguarded. Hence, in both cases, a contradiction is reached. These two cases are displayed in Figure \ref{fig-P3Cn011}.\end{proof}

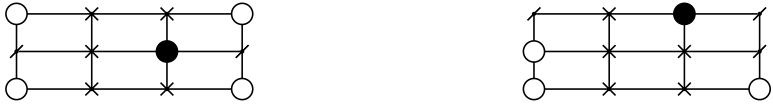
\begin{figure}[h!]\begin{center}
\begin{minipage}{0.4\textwidth}\begin{tikzpicture}[smallgraph]
\gridgraph{3}{4}
\inSa{1}{1} \inSa{1}{3} \inS{3}{2} \inSa{4}{1} \inSa{4}{3}
\notSa{1}{2} \notS{2}{1} \notS{2}{2} \notS{2}{3} \notS{3}{1} \notS{3}{3} \notSa{4}{2}
\end{tikzpicture}
\end{minipage}\hspace*{1cm}\begin{minipage}{0.4\textwidth}\begin{tikzpicture}[smallgraph]
\gridgraph{3}{4}
\inSa{1}{2} \inSa{1}{3} \inS{3}{1} \inSa{4}{3}
\notSa{1}{1} \notS{2}{1} \notS{2}{2} \notS{2}{3} \notS{3}{2} \notS{3}{3} \notSa{4}{1} \notSa{4}{2}
\end{tikzpicture}\end{minipage}
\caption{The two cases considered in Lemma \ref{lem-P3Cn011}.\label{fig-P3Cn011}}\end{center}\end{figure}

\begin{lem}Suppose $S$ is a secure dominating set for $P_3 \Box C_k$, $k = 1\mbox{ mod }3$, such that $|S| = k+1$, but that $\gamma_s(P_3 \Box C_{k-3}) = k-1$. If $k \geq 29$, then $m_0 \geq 2$.\label{lem-P3Cnm02}\end{lem}

\begin{proof}Suppose $m_0 = 0$. By Lemma \ref{lem-P3Cnm3}, $m_3=0$, and then equation \ref{eq-balance} implies that $m_1 = k-1$ and $m_2 = 1$. This implies there are $k-1$ type 1 copies in a row, and from Lemma \ref{lem-P3Cn1s} this means that $k \leq 9$, contradicting the initial assumption.

Then, suppose $m_0 = 1$. Then since $m_3 = 0$, we have $m_1 = k-3$, $m_2 = 2$. Then, it can be seen from Lemma \ref{lem-P3Cn011} and Lemma \ref{lem-P3Cnm02} that $k \leq 14$, which is also a contradiction.\end{proof}

We are now ready to prove the main theorem of this section.

\begin{thm}Consider $P_3 \Box C_n$ for $n \geq 3$. Then,\label{thm-P3Cn}
$$\gamma_s(P_3 \Box C_n) = \left\{\begin{array}{lcc}n + 1 & & \mbox{ if } n = 4, 7,\\
3 \left\lceil \frac{n}{3} \right\rceil & & \mbox{otherwise.}\\
\end{array}\right.$$\end{thm}

\begin{proof}We use the exact formulation from \cite{burdett} to confirm Theorem \ref{thm-P3Cn} for $n \leq 26$. Now, suppose there is some value $k \geq 27$ such that Theorem \ref{thm-P3Cn} is true for $n = 3, \hdots, k-1$, but not for $n = k$.

If $k = 1\mbox{ mod }3$ then from Lemma \ref{lem-P3Cnm02}, we have $m_0 \geq 2$, and hence there are blocks in $S$. Furthermore, from Lemma \ref{lem-P3Cnm3}, we have $m_3=0$, and then equation \ref{eq-balance} implies that $m_2 = m_0 + 1$. Hence, there is a block containing at least two copies of type 2. Suppose this block is of length $l$. Then we can trim out this block, removing at least $l$ entries from $S$. Using Corollary \ref{cor-trim}, this is a contradiction unless $k-l \in \{4,7\}$ and there are exactly two type 2 copies in the block. However, by Lemma \ref{lem-P3Cn1s}, there can be at most eight type 1 copies in a row and also by Lemma \ref{lem-P3Cn011}, $S$ cannot contain the pattern 011. Such a block could only have maximum length of 14 and since $k \geq 27$, this is impossible, and so $k \neq 1\mbox{ mod }3$.

Suppose $k = 2\mbox{ mod }3$. Then we have $\gamma_s(P_3 \Box C_k) \leq k$, and a secure dominating set $S$ exists such that $|S| = k$. Suppose that there is any type 0 copy in $S$. Then from Lemma \ref{lem-smoothing} we have $|S| \geq \gamma_s(P_3 \Box C_{k-1})$. However, since $n = k$ is the first time that Theorem \ref{thm-P3Cn} is not true, this implies that $|S| \geq k+1$, which is a contradiction. Hence, there must not be any type 0 copies in $S$. Then equation (\ref{eq-balance}) implies that every copy is of type 1. However, from Corollary \ref{cor-P3Cn0mod3}, this implies that $k = 0\mbox{ mod }3$ which is a contradiction. Hence, $k \neq 2\mbox{ mod }3$.

The only remaining possibility is that $k = 0\mbox{ mod }3$. Then we have $\gamma_s(P_3 \Box C_k) \leq k-1$, and a secure dominating set $S$ exists such that $|S| = k-1$. This implies that at least one type 0 copy exists in $S$. From Lemma \ref{lem-smoothing} we have $|S| \geq \gamma_s(P_3 \Box C_{k-1})$. However, since $n = k$ is the first time that Theorem \ref{thm-P3Cn} is not true, this implies that $|S| \geq k$, which is a contradiction, completing the proof.\end{proof}


\end{document}